\documentclass[12pt,reqno]{amsart}
\usepackage{geometry}                
\usepackage{multirow}

\usepackage{graphicx}
\usepackage{amsmath,amssymb,amsthm}
\usepackage{epstopdf}
\usepackage{setspace}
\usepackage[authoryear]{natbib}
\usepackage{color}
\usepackage{comment}

\theoremstyle{plain}

\newtheorem{theorem}{Theorem}[section]
\newtheorem{lemma}[theorem]{Lemma}
\newtheorem{assumption}{Assumption}[section]
\theoremstyle{remark}

\newtheorem{algo}{Algorithm}[section]
\newtheorem{remark}{Remark}[section]
\newcommand{\vecc}{\mathsf{vec}}
\newcommand{\Var}{\mathsf{Var}}
\newcommand{\E}{\mathsf{E} }
\newcommand{\RR}{\mathbb{R}}
\newcommand{\rank}{\mathsf{rank}}
\global\long\def\Scal{\mathcal{S}}%
\newcommand{\Cov}{\mathsf{Cov}}

\numberwithin{figure}{section}
 \numberwithin{equation}{section}
 \numberwithin{table}{section}
\begin{document}

\title{Inference for   Low-Rank Models} 

\author{Victor Chernozhukov}
\address{Department of Economics, MIT,  Cambridge, MA 02139}
\email{vchern@mit.edu}

\author{Christian Hansen}
\address{Booth School of Business, University of Chicago, Chicago, IL 60637}
\email{Christian.Hansen@chicagobooth.edu}
 
\author{Yuan Liao}
\address{Department of Economics, Rutgers University, New Brunswick, NJ 08901}
\email{yuan.liao@rutgers.edu}

\author{Yinchu Zhu}
\address{Department of Economics, Brandeis University, 415 South St, Waltham, MA 02453 }
\email{yinchuzhu@brandeis.edu}

\date{This draft: \today}                                           

\begin{abstract}

This paper studies inference in linear models with a high-dimensional parameter matrix that can be well-approximated by a ``spiked low-rank matrix.'' A spiked low-rank matrix has rank that grows slowly compared to its dimensions and nonzero singular values that diverge to infinity. We show that this framework covers a broad class of models of latent-variables which can accommodate matrix completion problems, factor models, varying coefficient models, and heterogeneous treatment effects. For inference, we apply a procedure that relies on an initial nuclear-norm penalized estimation step followed by two ordinary least squares regressions. We consider the framework of estimating incoherent eigenvectors and use a rotation argument to argue that  the eigenspace  estimation is asymptotically unbiased. Using this framework we  show that our procedure provides asymptotically normal inference and achieves the semiparametric efficiency bound. We illustrate our framework by providing low-level conditions for its application in a treatment effects context where treatment assignment might be strongly dependent. 
\end{abstract}

\maketitle

 
 
\onehalfspacing

 \newpage

 \section{Introduction} 
 
 We study  inference for linear low-rank models:
 $$
 Y= X \circ \Theta + \mathcal E
 $$
 where $(Y,X,\Theta,\mathcal E)$ are $n\times p$ matrices with both $n,p\to\infty$ and $\circ$ denotes the matrix element-wise product. We observe data $(X,Y)$, and $\mathcal E$ represents unobserved statistical noise. The model parameter is the matrix coefficient $\Theta.$ We assume $\Theta$ follows an \textit{approximate spiked low rank model}: $\Theta$ can be well-approximated by a low-rank matrix whose rank $J$ is either fixed or grows slowly compared to $n,p$ and whose largest $J$ singular values diverge with $(n,p)$. Our main goal is performing statistical inference on both sparse and dense linear combinations of elements of $\Theta$.
 
 Under the approximate spiked low rank model structure, nuclear-norm regularization provides a natural benchmark approach to estimating $\Theta$. There is a substantial literature that studies rates of convergence of nuclear-norm penalized estimators; see, e.g., \cite{koltchinskii2011nuclear} and \cite{negahban2011estimation} for prominent examples. 
 Providing results in low-rank models that allow for ready construction of inferential objects such as confidence intervals has been a topic in the more recent literature. For example, \cite{xia2019statistical} and \cite{chen2019inference} study inference in settings where the matrix parameter of interest has an exact low-rank structure with fixed rank and elements of $X$ are i.i.d. copies from an unknown distribution.
 
 We contribute to this literature by establishing asymptotic normality for low-rank estimators. Our method starts with an initial estimator of $\Theta$ obtained using nuclear-norm regularization from which we extract the right singular vectors. 
 We then treat the extracted singular vectors as data and obtain estimates of the left singular vectors and updated estimates of the right singular vectors by applying additional least squares steps. The final estimator is then the product of the estimated left and right singular vectors. We make use of a rotation argument to show that, in terms of estimating the space of the singular vectors, the regularization bias of the first step nuclear-norm penalized estimation aligns with the space spanned by the true singular vectors. Thus, the regularization bias is effectively eliminated by the additional least squares steps. 
 
 We prove that our estimator for linear functionals of the low-rank matrix is asymptotically normal. We also establish the semiparametric efficiency bound and show that our estimator attains the efficiency bound.   The notion of semiparametric efficiency in the presence of high-dimensional nuisance parameters is adopted from \cite{jankova2018semiparametric}. 
 Our result is novel relative to \cite{jankova2018semiparametric} because they deal with sparse models while we look at linear combinations of a low-rank matrix. 
 
 Our conditions allow possible strong dependence within $X$, which is useful in many contexts. For example, in the matrix completion context, we can accommodate persistence in observed and missing entries rather than relying on independent missingness. In our treatment effects example, allowing strong dependence allows us to consider scenarios where units are first in the control state for a period of time and then enter the treated state and remain there until the end of the sample period.

 We rely on two key technical conditions in establishing asymptotic normality of our proposed estimator. We first assume $\Theta$ has \textit{spiked singular values} (SSV), which requires that the non-zero singular values are large. This condition ensures that the rank of $\Theta$ can be consistently estimated and that the singular vectors are estimated sufficiently well for use in Stage 2 of the procedure. In the inference context, to make \textit{entrywise} inference for the low-rank matrix, the SSV condition on the singular values seem necessary. 
 
 The second condition relates to \textit{incoherent singular vectors} as defined in, e.g., \cite{candes2009exact,candes2010matrix, keshavan2010matrix}; and \cite{chen2019noisy}. This condition requires that the signals on the singular vectors should be approximately evenly distributed across their entries. Under the incoherence condition, we use a ``rotation" argument to show that our approach provides asymptotically unbiased estimates of the eigenvector space. We note that incoherence does rule out the setting of ``sparse PCA,'' which needs a separate treatment and often requires explicit debiasing steps as in, e.g., \cite{jankova2021biased}.

 We note that the SSV and incoherence conditions are strong and are often absent in the literature when probability bounds are derived. However, probability bounds for the Frobenius risk in general cannot imply the asymptotic distribution of estimators. In particular, one of the objects of interest in this paper is to make inference on sparse linear combinations of rows (or columns) of $\Theta$, including elementwise inference. Recent developments for perturbation bounds of entrywise eigen-analysis require SSV to make \textit{entrywise} inference for a low-rank matrix; see, e.g., \cite{abbe2020entrywise}.  
 
 To further illustrate that both conditions seem necessary for good performance, we provide new minimax theory on convergence rates without them. These results verify that it is impossible to guarantee entrywise consistency without SSV or incoherence and show that,  for dense linear combinations, the optimal rates one can achieve without these conditions are potentially much worse than those available under them. Finally, as these minimax rates do show that inference for dense linear combinations may proceed without SSV or incoherence, we provide an alternative inference for dense functionals when these conditions are relaxed.

 \noindent \textbf{The literature}
 
 \medskip
 
 Low-rank regression has been extensively studied in the literature. Much of this work focuses on deriving sharp deviation bounds for low-rank estimators; see, for instance, \cite{recht2011simpler}; \cite{gross2011recovering}; \cite{rohde2011estimation}; \cite{koltchinskii2011nuclear}; \cite{dray2015principal}; \cite{zhu2019high}; \cite{candes2010matrix}; \cite{hastie2015matrix}; \cite{keshavan2010matrix}; and \cite{sun2012calibrated}. As with \cite{xia2019statistical} and \cite{chen2019inference}, our paper complements the literature  by providing asymptotic distributional results. A key difference between our work and \cite{xia2019statistical} and \cite{chen2019inference} is that we adopt a different approach that does not rely on explicit debiasing steps to achieve asymptotic normality. Rather, we rely on the fact that $\Theta$ is a ``product parameter'' obtained by multiplication of left and right singular vectors to verify that our procedure produces sufficiently regular estimators for asymptotic normality to hold without explicit debiasing.
 
 Our paper is related to \cite{chernozhukov2018inference} which considers inference in linear panel data models with multivariate coefficient matrices that admit factor structures.  There are several important differences between the two papers. Because \cite{chernozhukov2018inference} consider estimation of multiple matrix parameters, they employ a complicated orthogonalization step to deal with the fact that regularization bias in any of the matrix parameters spills over and impacts estimation of all other matrix parameters. \cite{chernozhukov2018inference} also rely on strong conditions on regressors while our conditions allow the regressor in our model to be strongly persistent. This generalization allows us to handle matrix completion problems with ``systematic missingness." We also provide several new optimality results to establish semiparametric efficiency and minimaxity. Finally, we explicitly allow for the low-rank structure to be an approximation by accounting for approximation errors and allowing the rank of the approximating low-rank structure to increase with the sample size. Admitting these characteristics broadens the applicability of the method. For example, many nonparametric models cannot be formulated as exact low-rank models with fixed rank but can be \textit{approximated} with low-rank models with slowly growing rank.
 
 Throughout the paper, we denote the maximum and minimum singular values of a matrix $A$ as $\psi_{\max}(A)$ and $\psi_{\min}(A)$. We use $\psi_j(A)$ to denote the $j^{\text{th}}$ largest singular value of $A$. We use $\|A\|_F$, $\|A\|$ and $\|A\|_{(n)}=\sum_{k=1}^{\min\{n,p\}}\psi_k(A)$ to respectively denote the matrix Frobenius norm, operator norm, and nuclear norm. We let $\|A\|_{\max}=\max_{ij}|(A)_{ij}|$ be the elementwise norm.  Let $\vecc(A)$ denote the vector that stacks the columns of $A$. For two stochastic sequences, we write $a_n\asymp b_n$ if $a_n=O_P(b_n)$ and $b_n=O_P(a_n)$, which means $a_n/b_n=O_P(1)$. Finally, $a\vee b$ means $\max(a,b)$.

 \section{Estimation Procedure}\label{sec:inference}

 \subsection{Spiked low-rank matrices}
 Consider the following model
 \begin{equation}\label{eq2.1}
 	y_{ij} = x_{ij}\theta_{ij}+\varepsilon_{ij},\quad i\leq n,\quad j\leq p,
 \end{equation}
 where we observe data $(y_{ij}, x_{ij})$ and $\varepsilon_{ij}$ is the noise term. Let $(Y, X, \Theta, \mathcal E)$ denote the $n\times p$ matrices of $(y_{ij}, x_{ij}, \theta_{ij}, \varepsilon_{ij})$. Then the matrix form of \eqref{eq2.1} is 
 $$
 Y= X \circ \Theta +\mathcal E
 $$
 where $\circ$ denotes the matrix element-wise product. The goal is to make inference about linear combinations of elements of $\Theta$. Throughout the paper, we impose that $\Theta$ and its associated singular values/vectors are random.
 Suppose  $\Theta$ can be decomposed as 
 \begin{equation}\label{eq2.2}
 	\Theta = \Theta_0 + R, 
 \end{equation}
 where $\Theta_0$ and $R$ are $n\times p$ matrices satisfying the following conditions:
 
 (i) $\Theta_0$ is a rank $J$ matrix where $J$ is either bounded or grows slowly compared to $(n,p)$. In addition, the nonzero singular values of $\Theta_0$ are ``spiked": 
 $$
 \psi_1(\Theta_0)\geq...\geq\psi_J(\Theta_0)\geq \psi_{np},\quad \psi_j(\Theta_0)=0 \ \forall \ j>J
 $$
 for some sequence $\psi_{np}\to\infty$. 
 
 (ii) $R$ is the low-rank approximation error whose entries $r_{ij}$ satisfy
 $$
 \max_{i, j} |r_{ij}| \leq O_P( r_{np})
 $$
 for some sequence $r_{np}\to 0$.

 (iii) Let $U_0=[u_1,...,u_n]'$ and $V_0=[v_1,...,v_p]' $ respectively denote the $n \times J$ and $p \times J$ matrices that collect the left singular vectors and right singular vectors of $\Theta_0$ corresponding to the nonzero singular values. We assume incoherent singular-vectors: 
 $$
 \max_{j\leq p}\|v_j\| =O_P(\sqrt{Jp^{-1}}),\quad \max_{i\leq n}\|u_i\| =O_P(\sqrt{Jn^{-1}}).
 $$

 Given the approximate low-rank structure of $\Theta$, a natural estimation strategy is nuclear-norm penalized optimization:
 \begin{equation}\label{eq3.1}
 	\widetilde\Theta=\arg\min_{\Theta\in\mathcal A} \|Y - X \circ \Theta \|_F^2+ \nu \|\Theta\|_{(n)}
 \end{equation}
 where $\mathcal A=\{\Theta: \|\Theta\|_{\max}\leq M\}$ and $\nu$ is a tuning parameter. Imposing the $\max$-norm constraint with a large constant $M>0$ helps stabilize the solution; see, e.g., \cite{klopp2014noisy}. Nuclear-norm penalized regression is natural as the solution is easy to compute. Another option would be to explicitly penalize the matrix rank; however, obtaining the solution to the rank penalized problem is in general difficult unless all elements of $X$ are equal to one. 
 Statistical properties of (\ref{eq3.1}), focusing on the minimax rate for $\|\widetilde\Theta-\Theta\|_F$, have been well-studied in the literature; see, e.g., \cite{koltchinskii2011nuclear} and \cite{negahban2011estimation}. It is also well-known that the singular values of $\widetilde\Theta$ suffer from shrinkage biases, and dealing suitably with this shrinkage bias is an important component of our inferential procedure.
 
 Finally, we assume that $J$, the rank of the low-rank component $\Theta_0$, is known for simplicity. For instance, in the treatment effect study where the parameter matrix is approximated by a low-rank structure via a sieve representation, the rank equals the sieve dimension which could be pre-specified. 
 In cases where rank is unknown, it can be consistently estimated. For example, one can apply the singular value thresholding method where the cut-off value for ``large singular values'' can be chosen to dominate the noise level; see, e.g., \cite{onatski2010determining} and \cite{fan2022estimating}.  
 
 \subsection{The Proposed Estimation Procedure}
 
 Let the singular value decomposition of $\Theta_0$ be
 $$
 \Theta_0= U_0 D_0 V_0':= \Gamma_0V_0',\quad \Gamma_0:= U_0D_0.
 $$
 Here, $D_0$ is the $J\times J$ diagonal matrix containing the nonzero singular values of $\Theta_0$, and $U_0$ and $V_0$ are respectively the $n\times J$ left singular vector matrix of $\Theta_0$ and $p\times J$ right singular vector matrix of $\Theta_0$ corresponding to the non-zero singular values. Let $\gamma_{0,i}'$ for $i = 1,...,n$ denote the rows of $\Gamma_0$, and let $v_{0,j}'$ for $j = 1,...,p$ denote the rows of $V_0$.
 
 \begin{algo}\label{alg3.1} Fix $i\leq n$. Estimate $\theta_{ij}$ ($j=1,...,p$) as follows:
 	\begin{description}
 		\item[Step 1] \textit{Sample splitting.} Randomly split the sample into $ \{1,...,n\}\backslash \{i\}=\mathcal I \cup \mathcal I^c $ disjointly, so that $|\mathcal I|_0= \lfloor (n-1)/2\rfloor$. Let $$\mathcal G_{\mathcal I} := (Y_{\mathcal I}, X_{\mathcal I}, \Theta_{\mathcal I} )$$ respectively denote the $|\mathcal I|_0\times p$ submatrices of $ \mathcal G := (Y, X, \Theta)$ for observations $i\in \mathcal I$. Estimate the low-rank matrix $\Theta_{\mathcal I}$ as
 		\begin{equation} 
 			\widetilde\Theta_{\mathcal I}=\arg\min_{\|\Theta_{\mathcal I}\|_{\max}<M} \|Y_{\mathcal I} - X_{\mathcal I}\circ\Theta_{\mathcal I}\|_F^2+ \nu \|\Theta_{\mathcal I}\|_{(n)}.
 		\end{equation}
 		We provide a specific feasible choice for $\nu$ when discussing the simulation example in Section \ref{sec:sim}.
 		Let $\widetilde V_{\mathcal I}=(\widetilde v_1,...,\widetilde v_p)'$ be the $p\times J$ matrix whose columns are the first $J$ eigenvectors of $ \widetilde\Theta_{\mathcal I}' \widetilde\Theta_{\mathcal I}$. 
 		\item[Step 2] \textit{Unbiased estimate of $\Gamma_0, V_0$.} Using data $\mathcal I^c$, obtain
 		$$
 		\widehat\gamma_{k, \mathcal I}=\arg\min_{\gamma} \sum_{j=1}^p[y_{kj}- x_{kj}\cdot \gamma'\widetilde v_j]^2,\quad k\in \mathcal I^c\cup\{i\}.
 		$$
 		Update estimates of $V_0$ as $\widehat V_{\mathcal I}=(\widehat v_{1,\mathcal I},...,\widehat v_{p,\mathcal I})'$, where $$
 		\widehat v_{j,\mathcal I}=\arg\min_{v} \sum_{k\in \mathcal I^c\cup\{i\}}[y_{kj}- x_{kj}\cdot \widehat \gamma_{k, \mathcal I}'v]^2,\quad j=1,...,p.
 		$$
 		\item[Step 3] \textit{Exchange $\mathcal I$ and $\mathcal I^c$.} Repeat Steps 1-2 with $\mathcal I$ and $\mathcal I^c$ exchanged to obtain 
 		$\widehat\gamma_{k,\mathcal I^c}$ for $k\in\mathcal I\cup \{i\}$ and $ \widehat V_{\mathcal I^c}$. Define the estimator for $\theta_{ij}$ as
 		$$
 		\widehat\theta_{ij} = \frac{1}{2} [\widehat\gamma_{i,\mathcal I}' \widehat v_{j,\mathcal I}+\widehat\gamma_{i,\mathcal I^c}' \widehat v_{j,\mathcal I^c}] . 
 		$$
 	\end{description}
 \end{algo}

 We only iterate least squares \textit{once} in Step 2. The least squares steps following the use of nuclear-norm penalized estimation are analogous to approaches in the sparse regression setting that rely on refitting the least squares using selected regressors in a first step, such as post-lasso, e.g. \cite{belloni2013least}. The motivation is similar in wanting to alleviate shrinkage biases induced in the initial penalized estimation step. In addition, we split the sample $ \{1,...,n\}\backslash \{i\}=\mathcal I \cup \mathcal I^c $ so that $i$ is excluded from both subsamples. Splitting in this way ensures that the $\varepsilon_{ij}$ for the $i$ of interest are independent of observations in both subsamples assuming independence across $i$. 
 
 Stage 2, which involves two least squares estimation steps, is the essential stage to alleviating shrinkage bias. It starts with treating $\widetilde V$ from the penalized regression as observed data. A key ingredient of the analysis is to establish that this step produces an approximately unbiased estimator $\widehat\Gamma$, which then allows construction of a well-behaved estimator $\widehat\Theta$ in the final step. Given its importance, we provide the intuition for this step in Section \ref{sec:discussion}.

 \begin{remark}
 	The proposed procedure is similar to the ``alternating minimization" (AltMin) method in the literature, e.g., \cite{hastie2015matrix} and \cite{jain2013low}. There are two key differences. The first is that the AltMin procedure would iterate until convergence. In contrast, we only iterate once and good asymptotic statistical properties are guaranteed. The second difference is that penalization is often carried throughout iterations in the AltMin procedure. Thus, AltMin-type estimators have asymptotic shrinkage biases, which complicates establishing asymptotic normality. By employing unpenalized least squares in Stage 3, our procedure ensures the final estimator does not have large shrinkage bias asymptotically. 
 \end{remark}

 \section{Discussion}\label{sec:discussion}

 We make use of a ``rotation" argument and the structure of the low-rank matrix parameter to prove that eigenspace estimation is approximately unbiased if singular vectors are incoherent. Before turning to the matrix parameter setting, we introduce the main idea in the context of estimating a scalar parameter that is itself a product of two parameters.

 \subsection{Inference about Product Parameters}\label{subsub: product}
 
 Consider the problem of estimating a scalar parameter $\theta$ that can be written as the product of another two scalar parameters:
 $$
 \theta= \gamma \beta,\quad \gamma, \beta\in\mathbb R.
 $$
 Suppose some initial estimate $\widetilde\beta$ can be obtained for $\beta$, which is consistent but may have first-order bias. 
 In addition, suppose that $\theta$ can be identified as the unique minimizer of a population loss function: 
 $$
 \theta =\arg\min_{a\in \mathcal A} Q(a)
 $$
 where $Q(\cdot)$ is the loss function and $\mathcal A$ is the parameter space. Let $Q_n(\cdot)$ denote the sample version of $Q$ and suppose both $Q_n$ and $Q$ are twice continuously differentiable. Let $\dot{Q}_n(a)=\frac{d}{da}Q_n(a)$, $\dot{Q}(a)=\frac{d}{da}Q(a)$, and $\ddot{Q}(a)= \frac{d^2}{da^2}Q(a)$. Let $(\gamma,\beta,\theta)$ represent the true values of the parameters.
 
 We consider an iterative procedure to estimate $\theta$ that mimics the approach we propose in the matrix parameter setting:
 
 \smallskip
 (i) Obtain $\widehat\gamma=\arg\min_{\gamma} Q_n(\gamma \widetilde\beta)$
 
 (ii) Obtain $\widehat\beta=\arg\min_{\beta} Q_n(\widehat\gamma \beta)$
 
 (iii) Set $\widehat\theta= \widehat\gamma\widehat\beta $.
 
 \smallskip
 In step (i), standard analysis based on Taylor expansion leads to
 \begin{equation}\label{eq3.2}
 	\widehat\gamma -\gamma
 	=G^{-1} \beta\dot{Q}_n(\theta)
 	+G^{-1} \partial^2_{\gamma,\beta} Q(\gamma\beta)(\widetilde\beta-\beta)
 	+o(|\widehat\gamma-\gamma|),\footnote{We have $0=\partial_{\gamma} Q_n(\widehat\gamma \widetilde\beta)=\partial_{\gamma} Q_n(\gamma \beta) 
 		+\partial^2_{\gamma,\beta} Q_n(\gamma \beta)(\widetilde\beta-\beta)
 		+\partial^2_{\gamma,\gamma} Q_n(\gamma \beta)(\widehat\gamma-\gamma)+O(|\widehat\gamma-\gamma|^2+|\widehat\beta-\beta|^2) $, and $\partial_{\gamma} Q_n(\gamma \beta) = \beta\dot Q_n(\theta)$. Inverting $\partial^2_{\gamma,\gamma} Q_n(\gamma \beta)$ leads to (\ref{eq3.2}).}
 \end{equation}
 where $G=-\partial^2_{\gamma,\gamma}Q_n(\gamma\beta)$. The first term in the expansion is the score which leads to asymptotic normality in usual cases. The second term reflects the effect of the initial estimate $\widetilde\beta$. 
 
 {In general, the second term will lead to poor performance of $\widehat\gamma$ if the initial estimator $\widetilde\beta$ is ill-behaved. 
 	One approach, dating back to at least \cite{Neyman59}, is to rely on estimation strategies make use of appropriately ``orthogonalized'' scores. This property would correspond to basing estimation on an objective function that satisfied $\partial^2_{\gamma,\beta} Q(\gamma\beta)=0$ at the population level in the present case.} See, e.g., \cite{CHS:AnnRev} for a review of such approaches.

 The fact that the ``product parameter'' $\theta$, rather than $\gamma$ itself, is the object of interest allows a new argument in this paper. The key is that the loss function depends on $\theta$ only through the product of $(\gamma,\beta)$. It is straightforward to verify that 
 $$\partial_{\gamma,\beta}^2 Q(\gamma\beta) 
 =\gamma\ddot{Q}(\theta) \beta + \underbrace{\dot{Q}(\theta)}_{\text{score}=0}
 =\gamma\ddot{Q}(\theta) \beta.
 $$
 Substituting this expression for $\partial_{\gamma,\beta}^2 Q(\gamma\beta)$ into (\ref{eq3.2}) then produces 
 $$
 \widehat\gamma -\gamma =G^{-1} \beta\dot{Q}_n(\theta) +G^{-1} \gamma \ddot{Q}(\theta)\beta(\widetilde\beta-\beta) +o(|\widehat\gamma-\gamma|).
 $$
 An important observation is that the second term $G^{-1} \gamma \ddot{Q}(\theta)\beta(\widetilde\beta-\beta)$ is proportional to $\gamma$. We can move it to the left-hand-side of the expansion for $\widehat\gamma$ to obtain
 $$
 \widehat\gamma - H\gamma = G^{-1} \beta\dot{Q}_n(\theta)
 +o(|\widehat\gamma-\gamma|)
 $$
 for $H:= 1+ G^{-1} \ddot{Q}(\theta)\beta(\widetilde\beta-\beta).$ Hence, $\widehat\gamma$ estimates a ``rotated'' version of $\gamma$ with no first-order bias. As such, in the sense of estimating the ``space'' of $\gamma$, the effect $\widetilde\beta-\beta$ is negligible as it is ``absorbed" by the rotation matrix. In addition, $H$ is asymptotically invertible since $H\to^P 1$.

 Moving on to step (ii),   it is clear that $\widehat\beta$ estimated in this step will be an approximately unbiased estimator for $H^{-1}\beta$. The rotation matrices will then cancel in estimating the parameter of interest:
 $$\widehat\theta:=\widehat\gamma\widehat\beta
 =\gamma HH^{-1}\beta +o_P(1) = \theta + o_P(1).
 $$
 After appropriate scaling, the leading term hidden in the $o_P(1)$ in the final expression will also be asymptotically normal. It is this cancellation of rotation matrices that underlies our ``rotation-unbiasedness". Furthermore, in models where $\sqrt{n}$-consistency is attainable, $\sqrt{n}(\widehat\theta-\theta)$ is asymptotically normal as long as the initial estimator satisfies $|\widetilde\beta-\beta|=o_P(n^{-1/4})$.

 The intuition of ``rotation-unbiasedness" as described above has also been observed previously in the literature. \cite{keshavan2010matrix} studied local geometric properties in Grassmann manifold and related optimization algorithms. \cite{sun2016guaranteed} examined the local geometry of the loss $f(\Gamma, V)=\|Y-\Gamma V'\|_F^2$ in the matrix completion context.  Our observation aligns with theirs, but we use this observation in the context of estimation bias. In our setting, the geometry of product-parameter $\gamma\beta$ ensures that the effect of first-step estimation error $\widetilde\beta-\beta$ is aligned with the space of the true $\gamma$. This alignment results in our ability to establish asymptotic normality of our final estimator without relying on any additional debiasing schemes beyond the use of a single set of least squares steps in Step 2 of our algorithm. 
 
 \subsection{Eigenspace estimation}
 
 In the low-rank inference context, recall that $\Theta_0=\Gamma_0 V_0',$ which is the product of two parameters. Related to the simple example in the previous section, we think about $V_0$ as $\beta$ and use the singular vectors $\widetilde V$ extracted from the nuclear-norm regularized estimator as its initial estimate.

 Write $\widehat\Gamma=(\widehat\gamma_1,...,\widehat\gamma_n)'$ and $\widetilde V=(\widetilde v_1,...,\widetilde v_p)'$. Then for each $i\leq n$, 
 $$
 \widehat\gamma_i=\arg\min_{\gamma} Q_{i}(\gamma, \widetilde V),\quad Q_{i}(\gamma, \widetilde V):=\sum_{j=1}^p[y_{ij}- x_{ij}\cdot \gamma'\widetilde v_j]^2.
 $$
 Then for some $J\times J$ matrix $G^{-1}$, Taylor expansion leads to 
 $$
 \widehat\gamma_i-\gamma_i = G^{-1} \partial_\gamma Q_{i}(\gamma_i, V_0)
 + \frac{\partial^2Q_{i}(\gamma_i, V_0)}{\partial\gamma\partial\vecc(V)} \vecc(\widetilde V-V_0) + \text{higher order terms}.
 $$
 The leading term $G^{-1} \partial_\gamma Q_{i}(\gamma_i, V)$ is asymptotically normal if $V_0$ is incoherent. The second term satisfies
 \begin{eqnarray*}
 	\frac{\partial^2Q_{i}(\gamma_i, V)}{\partial\gamma\partial\vecc(V)} \vecc(\widetilde V-V_0) &=&H_1 \gamma_i + \Delta_{i}
 \end{eqnarray*}
 for some rotation matrix $H_1$ and higher order term $\Delta_{i}$.
 
 The term $H_1 \gamma_i$ is a rotated version of $\gamma_i$. Defining $H:= I+ H_1$ and moving $H_1\gamma_i$ to the left side then yields the matrix  form expansion:
 $$
 \widehat\Gamma-\Gamma_0 H= \partial_\Gamma Q_{p}(\Gamma, V_0) G^{-1} + \text{higher order terms}
 $$
 where $\partial_\Gamma Q_{p}(\Gamma, V_0) $ is an $n\times J$ matrix whose $i^{\text{th}}$ row is the transpose of $\partial_\gamma Q_{i}(\gamma_i, V_0)$.
 Following the logic outlined in Section \ref{subsub: product}, we have that the follow-up estimator $\widehat V$ will recover an appropriately rotated version of $V$ to cancel with $H$. Consequently,
 $$
 \widehat\Theta = \widehat\Gamma\widehat V' \approx \Gamma_0 HH^{-1}V_0'
 =\Gamma_0V_0'=\Theta_0.
 $$

 It will then follow that $\widehat\Theta$ is approximately unbiased with sampling distribution that can be approximated by a   centered Gaussian distribution. As in the simpler scalar case, the key feature we take advantage of is that we only need the estimated $V_0$ to have the same span as the actual $V_0$ if our goal is inference about $\Theta$ or the space spanned by the singular vectors. 
 
 \subsection{Sample Splitting }\label{subsec: sample split}
 Our argument for demonstrating that the higher-order term, $\Delta_i$ is asymptotically negligible relies on sample splitting. The structure of $\Delta_i$ is
 $$
 \Delta_{i}= B \sum_{j=1}^p(\widetilde v_j-v_j)\varepsilon_{ij} x_{ij}.
 $$
 for some matrix $B$. 
 
 For a fixed $i$, let $\mathcal I\subset\{1,...,n\}\backslash i $ be a subset of unit indexes that does not include $i$; and let 
 $$
 \mathcal D_{\mathcal I} =\{(y_{kj}, x_{kj}): k\in \mathcal I, j\leq p\}.
 $$
 Our approach uses only data $\mathcal D_{\mathcal I}$, rather than making use of the full data set, for the initial nuclear-norm penalized regression from which we extract singular vectors for the subsequent OLS rotation-debiasing step.  Maintaining independence across $i$, estimation errors in the initial estimator of the singular vectors are then independent of variables indexed by $i$ because $i\notin \mathcal I$. Assuming $ \varepsilon_{ij} $ is independent across subjects $i=1,..., n$, 
 we then have that $\varepsilon_{ij} x_{ij} $ is independent of estimation error in the singular vectors, $\widetilde v_j-v_j$. We can then easily argue that $ \Delta_{i}$ has no impact on the asymptotic distribution of the final estimator.

 \section{Asymptotic Results}\label{sec:4}
 
 We now present our main results. In Section \ref{sec:norm}, we lay out key conditions and state our result on asymptotic normality. We then provide a brief discussion of semiparametric efficiency in Section \ref{sub: efficiency} and then highlight the role of the key SSV and incoherence conditions in Section \ref{sub: SSV} where we present novel minimax results. Finally, we present an alternative estimation scheme for dense linear combinations in Section \ref{sub: dense no SSV}.

 \subsection{Asymptotic Normality}\label{sec:norm}
 
 The goal is to establish inferential theory for the linear functional $\theta_i'g$. Here $\theta_i'$ denotes the $i^{\text{th}}$ row of $\Theta$, and $g=(g_1,...,g_p)'\in\mathbb R^p$ is a vector of weights of interest with non-zero weights collected in 
 $$
 \mathcal G=\{j\leq p: g_j\neq 0\}.  
 $$ 
 Inference on a linear combination of a column of $\Theta$ can be carried out similarly by switching the roles of $i$ and $j$. 
 Two examples of $g$ are of particular interest: 
 
 \smallskip

 \textit{Sparse weights}: $g$ is a sparse vector with a bounded number of non-zero elements: 
 \begin{equation}\label{eq4.1sparse}
 	|\mathcal G|=O(1).
 \end{equation}
 $\theta_i'g$ thus corresponds to a linear combination of a small number of elements and may be used when we are particularly interested in just a few components of $\theta_i$. The sparse $g$ scenario includes $g=e_j$ where $e_j=(0,...,0,1,0,...,0)$ is the $j^{\text{th}}$ standard vector for a particular $j$ in which case $\theta_i'g=\theta_{ij}$. 
 
 \smallskip

 \textit{Dense weights}: $g$ is a dense vector, in the sense that $|\mathcal G| = O(p)$, but 
 \begin{equation}\label{eq4.1dense}
 	\max_{j\leq p}|g_j|<Cp^{-1},\quad \text{ for some } C>0.
 \end{equation}
 In this case, $\theta_i'g$ typically represents a weighted average of all components of $\theta_i$ and includes $g=(\frac{1}{p},...,\frac{1}{p})'$ as a special case.

 The following assumption formally quantifies the requirement of $g$. Consider the matrix of standardized right singular vectors:
 $$
 \bar V'= \sqrt{p}V_0'.
 $$

 \begin{assumption}\label{ass4.3}
 	For some constants $c,C>0,$
 	$$ c<\| \bar V'g\|\leq C,\quad \|g\|<C . $$ 
 	In addition, $g$ satisfies either (\ref{eq4.1sparse}) or (\ref{eq4.1dense}).
 \end{assumption}
 The next assumption restricts the noise data generating process (DGP).
 \begin{assumption}[DGP for $\varepsilon_{ij}$]\label{assdgp:NOISE}
 	(i) $\varepsilon_{ij}$ is conditionally independent across $i\leq n$ and $j\leq p$, given $(\Theta, X)$. Also,
 	$\E(\varepsilon_{ij}| \Theta,X)=0$ and $ \max_{ij}\E [\varepsilon_{ij}^4|\Theta, X]<C $ almost surely. 
 	(ii) At least one of the following holds:
 	
 	\begin{description} 
 		
 		\item[a] $\min_{ij}\Var( \varepsilon_{ij} |\Theta, X)>c.$
 		\item[b] $ \varepsilon_{ij}$ can be decomposed as $\varepsilon_{ij}= e_{ij} x_{ij} $ with $\min_{ij}\Var( e_{ij} |\Theta, X)>c.$
 	\end{description}
 \end{assumption}

 Assumption \ref{assdgp:NOISE} (ii) is stated in a way that specifically covers the well-known matrix completion problem:
 $$
 y_{ij}^*= \theta_{ij} + e_{ij}
 $$
 where $y_{ij}^*$ may not be observable, and $x_{ij}$ indicates the observability for each element. Then $\varepsilon_{ij}= e_{ij} x_{ij}$.

 The assumption below restricts the DGP of the design variable $x_{ij}$. The restrictions imposed are mild, and the assumption is stated so as to cover a variety of cases. Specifically, conditions (a)-(c) in Assumption \ref{assdgp} allow for various types of dependence among the $x_{ij}$.

 \begin{assumption}[DGP for $ x_{ij} $]\label{assdgp} 
 	(i) $\max_{ij}| x_{ij} |<C$ and $x_{ij}$ is independent of $\Theta$.
 	(ii) At least one of the following holds:
 	\begin{description}
 		\item[a] 
 		$ x_{ij}^2 $ does not vary across $i\leq n$. 
 		\item[b] $ x_{ij} ^2$ is independent across $(i,j)$.
 		In addition, $\E x_{ij}^2 $ does not vary with $i$. 
 		\item[c] $x_{ij}\in\{0,1\}$. Also, define
 		$
 		\mathcal B_i:=\{j\leq p: x_{ij} =1\} .
 		$
 		Then there is a set $\bar{\mathcal B}\subseteq\{1,...,p\}$, so that 
 		\begin{equation}\label{eq4.2daerv}
 			\max_{i\leq n}\sum_{j=1}^p1\{j\in \bar{\mathcal B} \vartriangle \mathcal B_i\} = o_P\left(d_{n,p}\right),\quad d_{n,p}:= \left(\frac{\min\{n,p, \psi_{np}\}p}{ (n+p)J + \|R\|_{(n)}^2}\right) J^{-(2+d+2b)}.\end{equation}
 		where $\bar{\mathcal B} \vartriangle \mathcal B_i =[\bar{\mathcal B} \cap \mathcal B_i ^c]\cup [\bar{\mathcal B}^c \cap \mathcal B_i]$ is the symmetric difference of two sets, and $d,b\geq0$ are constants defined in Assumption \ref{ass4.7} below.
 	\end{description}
 \end{assumption}

 Under Condition (ii).\textbf{a}, we can accommodate both conventional factor models by setting $x_{ij} = 1$ for all $i,j$ as well as conditional empirical factor models, where $x_{ij} = x_j$, with varying coefficients. An example of the latter is an asset pricing model with risk premia that vary across assets and over time where $x_j$ represents the common time-varying market factor. 
 
 Condition (ii).\textbf{b} could cover examples of PCA with missing data under heterogeneous missing probabilities as in \cite{zhu2019high}. In this case, we may take $j$ to represent subjects and $i$ to represent the index of repeated sampling within subject. The condition also accommodates scenarios where $x_{ij}$ represents a treatment indicator where random assignment of subjects $i$ to treatment states occurs independently in each period $j$. Such a structure may approximate some digital experimentation settings.

 Condition (iii).\textbf{c} allows for  some types of strong dependence in $x_{ij}$ across both $i$ and $j$ but restricts $x_{ij}$ to be binary as would be appropriate in missing data, matrix completion, and treatment assignment settings. In this condition, the set $\mathcal B_i$ represents unit-specific ``observation times" for unit $i$; and the set $\bar{\mathcal B}$ is common to all units. The quantity $\max_{i\leq n}\sum_{j=1}^n1\{j\in \bar{\mathcal B} \vartriangle \mathcal B_i\}$ thus measures the difference between the ``unit specific" observation times and the ``common" observation times. Condition (iii).\textbf{c} requires that these differences should be negligible. Hence,  all units should be observed at approximately the same time. For instance, suppose every unit is observed most of the time in the sense that 
 $$
 \max_{i\leq n} \sum_{j=1}^p1\{j: x_{ij}=0\} = o_P\left(d_{n,p}\right).
 $$
 Then Condition (iii).\textbf{c} holds with $\bar{\mathcal B}=\{1,...,p\}$. 

 Next, recall that $v_j$ and $u_i$ are respectively the $j^{\text{th}}$ right singular vector and the $i^{\text{th}}$ left singular vector of $\Theta_0$.

 \begin{assumption}[Incoherent singular vectors]\label{ass:4.6} 
 	$$ \E \max_{j\leq p}\|v_j\|^2= O (Jp^{-1}),\quad \E \max_{i\leq n}\|u_i\|^2=O ( Jn^{-1}).$$
 	
 \end{assumption}
 
 The incoherence condition ensures that information regarding the eigenspace  accumulates as the dimension increases and allows us to apply our ``rotation" argument to argue that estimating the eigenvector space is asymptotically unbiased. We provide low-level conditions that are sufficient for the incoherence condition in a treatment effects context where the low-rank matrix is formulated using nonparametric sieve representations in equation (\ref{eq5.3}).

 The next assumption places restrictions on various moments. 
 \begin{assumption}[Moment bounds]\label{ass.5mom}
 	There are matrices $A_i, B_j$ whose eigenvalues are bounded away from zero and infinity, so that 
 	$$
 	\max_{i\leq n}\|\sum_{j=1}^p x_{ij} ^2v_jv_j'- A_i\|=o_P(J^{-1/2}),\quad \max_{j\leq p} \|\frac{n}{|\mathcal S|}\sum_{i\in\mathcal S} x_{ij} ^2 u_i u_i' -B_j\|=o_P(J^{-1/2}).
 	$$
 	This should hold for $\mathcal S$ being sets $\{1,...,n\}, \mathcal I$ and $\mathcal I^c$. 
 \end{assumption}

 Finally, we present the required conditions on $\psi_{np}$, the signal strength of the non-zero singular values. Recall that $\psi_j(A)$ denotes the $j^{\textnormal{th}}$ largest singular value of $A$. We allow the eigengap to change with $J$, depending on constants $b, d \geq 0$. This generality complicates statement of the condition but is needed to accommodate settings where the rank $J$ is allowed to increase with sample sizes. We provide low-level conditions that are sufficient for the following assumption in the context of a treatment effect example in Lemma \ref{lem5.2}.

 \begin{assumption}[Signal-noise]\label{ass4.7} There are constants $b,d\geq 0$ such that

 	(i)  $\psi_{np}\leq \psi_J(\Theta_0)<\psi_1(\Theta_0)\leq O_P(J^b\psi_{np})$   for  a sequence $\psi_{np}\to\infty$ that satisfies  
 	$$
 	n^{-1/2} p J ^{7/2+2d+5b}+ (p\vee n)^{3/4} J^{5/4+d+2b} =o( \psi_{np} ). 
 	$$
 	(ii) Eigengap: There are $c, C>0$ and a sequence $\psi_{np}\to\infty$ so that with probability approaching one, 
 	$$
 	\psi_{j}(\Theta_0) - \psi_{j+1}(\Theta_0) \geq c\psi_{np}J^{-d},\quad j=1,..., J.
 	$$ 
 	(iii) The rank $J$ satisfies
 	$$
 	J^{3+2d+6b}=o_P( \min\{\sqrt{p},\sqrt{n}, p/\sqrt{n}\}). $$
 	(iv) The low-rank approximation error matrix $R=(r_{ij})_{n\times p}$ satisfies
 	$$ \max_{ij}|r_{ij}|^2(p\vee n)^{2} J^{3+4b}=o(1).
 	$$
 \end{assumption}

 \begin{theorem}\label{th4.1} 
 	Suppose $g$ is either dense or sparse, in the sense of (\ref{eq4.1sparse}) and (\ref{eq4.1dense}). Suppose Assumptions \ref{ass4.3}-\ref{ass4.7} hold, and the nuclear-norm tuning parameter satisfies $\nu>C(\sqrt{n+p})$ for some contant $C>0$. Then for a fixed $i\leq n$, 
 	$$
 	\frac{ \widehat\theta_i'g-\theta_i'g}{\sqrt{s_{np,1}^2+s_{np,2}^2}} \to^d N(0,1),
 	$$
 	where, with $ L_{j }=\sum_{i =1}^n x_{ij} ^2 \gamma_i \gamma_i'$ and $ 
 	\bar B=\sum_{j=1}^p(\E x_{ij} ^2) v_j v_j' $, 
 	\begin{eqnarray*}
 		s_{np,1}^2 &:=& 
 		\sum_{j=1}^p\sum_{t =1}^n \Var( \varepsilon_{tj} |\Theta, X)
 		[ \gamma_i' L_{j}^{-1} \gamma_{t} ]^2 x_{tj} ^2 g_j^2\cr
 		s_{np,2}^2&:=& 
 		\sum_{j=1}^p\Var( \varepsilon_{ij} |\Theta, X) x_{ij} ^2[v_j'\bar B^{-1} V_0'g]^2 .
 	\end{eqnarray*}
 \end{theorem}
 
 To estimate the asymptotic variance, we need to preserve the rotation invariance property of the asymptotic variance. We therefore estimate $\sigma_{np}^2$ separately within subsamples and produce the final asymptotic variance estimator by averaging the results across subsamples. We consider the homoskedastic case where $\Var(\varepsilon_{ij}|\Theta, X)= \sigma_j^2 $ for some constant $\sigma_j^2$, $j=1,...,p.$ In this case, standard errors can be estimated as
 \begin{eqnarray*}
 	\widehat s_{np,1}^2 &:=& 
 	\frac{1}{4}\sum_{j=1}^p\sum_{t \notin\mathcal I} \widehat\sigma_j^2
 	[\widehat \gamma_{i,\mathcal I}' \widehat L_{j,\mathcal I}^{-1} \widehat\gamma_{t} ]^2 x_{tj} ^2 g_j^2
 	+ \frac{1}{4}\sum_{j=1}^p\sum_{t \notin\mathcal I^c} \widehat\sigma_j^2
 	[\widehat \gamma_{i,\mathcal I^c}' \widehat L_{j,\mathcal I^c}^{-1} \widehat\gamma_{t} ]^2 x_{tj} ^2 g_j^2
 	\cr
 	\widehat s_{np,2}^2&:=& 
 	\frac{1}{2} \sum_{j=1}^p \widehat\sigma_j^2 x_{ij} ^2[\widetilde v_{j,\mathcal I}'\widehat B_{\mathcal I}^{-1} \widetilde V_{\mathcal I}'g]^2+\frac{1}{2} \sum_{j=1}^p \widehat\sigma_j^2 x_{ij} ^2[ \widetilde v_{j,\mathcal I^c}'\widehat B_{\mathcal I^c}^{-1} \widetilde V_{\mathcal I^c}'g]^2 \cr
 	\widehat\sigma_j^2&:=& \frac{1}{n} \sum_{t\notin\mathcal I} (y_{tj} - x_{tj}\cdot \widehat\gamma_{t,\mathcal I}' \widehat v_{j,\mathcal I}) ^2+ \frac{1}{n} \sum_{t\notin\mathcal I^c} (y_{tj} - x_{tj}\cdot \widehat\gamma_{t,\mathcal I^c}' \widehat v_{j,\mathcal I^c}) ^2
 \end{eqnarray*}
 where $\widehat L_{j,\mathcal I}=\sum_{t\notin\mathcal I} x_{tj}^2\widehat \gamma_t\widehat\gamma_t'$, 
 and $ \widehat B_{\mathcal I}=\sum_{j=1}^p x_{ij} ^2\widetilde v_{j,\mathcal I}\widetilde v_{j,\mathcal I}'
 $, and $\widehat L_{j,\mathcal I^c}$ and $\widehat B_{\mathcal I^c}$ are defined similarly. 
 
 It is interesting to note that  $\sigma^2_{np}:=s_{np,1}^2+s_{np,2}^2=O_P(\frac{1}{n}\|g\|^2+\frac{1}{p})$ in the case of fixed $J$. Thus, in this setting, the scaling of the asymptotic variance depends heavily on $\|g\|^2$. 
 
 \subsection{Semiparametric Efficiency}\label{sub: efficiency}

 The semiparametric efficiency bound for the case of sparse $g$ was established by \cite{chen2019inference} (Lemma 2) in matrix completion settings and by \cite{iwakura2014asymptotic} (Theorem 4.5) in pure factor models. Our asymptotic variance attains these previously established bounds if $e_{ij} $ is i.i.d. homoskedastic Gaussian, so we do not further discuss semiparametric efficiency in the sparse setting. 
 
 We now provide a semiparametric efficiency bound in the case of dense $g$ and verify that our estimator achieves this bound. For concreteness, suppose we are interested in 
 $h(\Theta)=\theta_1'g$ where $\theta_1'$ is the first row of $\Theta$ and $g$ is dense. In providing our result, we will allow for a wide range of distributions for $X_{1j}$ while maintaining the assumption that the error term is Gaussian to make calculation tractable. 
 
 Specifically, we suppose that $x_{ij}$ follows the distribution $f$ and $e_{ij}\sim N(0,\sigma^{2})$ are independent across $(i,j)$. Let $\mu_f=\E_{f} X_{1j}^2$. Under Assumptions \ref{ass4.3}-\ref{ass4.7}, the dominant term in the asymptotic variance is 
 \begin{align*}
 	s_{np,2}^{2} & =\sigma^{2}\sum_{j=1}^{p}x_{ij}^{2}[v_{j}'\bar{B}^{-1}V_0'g]^{2}\\
 	& =s_{*}^{2}(\Theta,f,\sigma)+o_P(s_{np,2}^{2}),\quad \text{where } s_{*}^{2}(\Theta,f,\sigma)=\sigma^{2}\mu_f^{-1}\|V_0'g\|^{2},
 \end{align*}
 and we also have $s_{np,1}^2=o_P(s_{np,2}^2)$. 
 Hence, $\widehat\theta_i'g-\theta_i'g=O_P(\|V_0'g\|)$ with asymptotic variance
 $$
 s_{np,1}^{2}+s_{np,2}^{2}=s_{*}^{2}(\Theta,f,\sigma)(1+o_{P}(1))
 $$
 in this case. 
 
 The following result verifies that $s_{*}^{2}(\Theta,f,\sigma)$ matches with the semiparametric efficiency bound. The notion of semiparametric efficiency in the presence of high-dimensional nuisance parameters is adopted from \cite{jankova2018semiparametric}. The idea is to derive the asymptotic Cram\'er-Rao bound for asymptotically unbiased estimators, and needs to be formally established in the high-dimensional setting. Our result is novel relative to \cite{jankova2018semiparametric} because they deal with sparse models and our setting has low-rank matrices as the nuisance parameters. 
 
 \begin{theorem}
 	\label{thm: semiparam bnd}
 	Consider  $h(\Theta)=\theta_1'g$, where $\theta_1'$ is the first row of $\Theta$ and $g$ is dense.  Let $x_{ij}\sim f$
 	and $e_{ij}\sim N(0,\sigma^{2})$ be independent across $(i,j)$. Define 
 	$$
 	\mathcal{M}=\left\{ (A,f,\sigma):\ \rank(A)\leq J,\ {Assumptions\ \ref{ass4.3}-\ref{ass4.7}\ hold}\right\} .
 	$$
 	Suppose that $T(Y,X)$ is an asymptotically unbiased estimator of $h(\Theta)$ in the sense that $\E_{(\Theta,f,\sigma)}T(Y,X)-h(\Theta)=o(s_{*}(\Theta,f,\sigma))$ where $\E_{(\Theta, f,\sigma)}$ denotes the expectation with respect to a given parameter $(\Theta, f,\sigma)$.  Then for any  sequence of $(\Theta,f,\sigma)\in\mathcal{M}$,
 	\[
 	\liminf_{n,p\rightarrow\infty}\frac{\E_{(\Theta,f,\sigma)}[T(Y,X)-h(\Theta)]^2}{s_{*}^{2}(\Theta,f,\sigma)}\geq1.
 	\]
 \end{theorem}
 

 \subsection{The role of spiked singular-values and incoherence}\label{sub: SSV}
 
 Two  key conditions that underlie our main results are the incoherence condition, Assumption \ref{ass:4.6}, and the spiked singular-value (SSV) condition, Assumption \ref{ass4.7}. We demonstrate the role of these conditions by providing minimax theory for estimating $\theta_i'g$ for a sparse or dense $g$ in a simple matrix completion problem where the missing indicators $x_{ij}$ are independent Bernoulli random variables without imposing SSV or incoherence.

 Define the following set of low-rank matrices
 \[
 \Scal=\left\{ A\in\RR^{n\times p}:\ \rank(A)\leq J\ {\rm and}\ \max_{1\leq i\leq n}\max_{1\leq j\leq p}|A_{ij}|\leq c_{1}\right\} 
 \]
 for a constant $c_{1}>0$ and for $J\geq1$. Here, $J$ is allowed to be either a fixed constant or a sequence tending to infinity.

 We prove the following result for matrix completion over the space $\Scal$. 
 Let $y_{ij}=x_{ij}\theta_{ij}+e_{ij}$, where $x_{ij}\sim{\rm Bernoulli}(\rho_{j})$
 and $e_{ij}\sim N(0,\sigma_{ij}^{2})$ are independent across $(i,j$).
 Suppose that there are constants $c_{2},...,c_{6}>0$ such that $\rho_{j}\in(c_{2},1-c_{2})$
 and $\sigma_{ij}\in(c_{3},c_{4})$ for any $(i,j)$.
 Let $\rho=(\rho_{1},...,\rho_{p})'$ and $\sigma=\{\sigma_{ij}\}_{1\leq i\leq n,\,1\leq j\leq p}$. In the theorem below, $T$ represents any measurable function of the data, typically regarded as an ``estimator" for $h(\Theta)=\theta_1'g.$

 \begin{theorem}[Minimax Rate]
 	\label{thm: minimax sparse}
 	Consider estimating $h(\Theta)=\theta_1'g=\sum_{j=1}^p\theta_{1j}g_j,$ and let $P_{(\Theta, f,\sigma)}$ denote the probability measure with respect to a given parameter $(\Theta, f,\sigma)$. We have the following results:
 	\begin{enumerate}
 		\item Sparse $g$: Let $g_1=1$ and $g_j=0$ for $j\geq 2$, i.e., $h(\Theta)=\theta_{11}$. Then 
 		\begin{equation}
 			\inf_{T}\sup_{\Theta\in\Scal}P_{(\Theta,\rho,\sigma)}\left(\left|T-h(\Theta)\right|>\kappa\right)>1/4, \label{eq: minimax sparse functional}
 		\end{equation}
 		where $\kappa>0$ is a constant depending on $(c_{1},c_{3})$
 		and $\inf_{T}$ is taken over all measurable functions of the data
 		$(X,Y)$.  
 		\item Dense $g$: Let $|g_{j}|\in[c_{5}/p,\ c_{6}/p]$ for all $j \in \{1,...,p\}$. Then 
 		\begin{equation}
 			\inf_{T}\sup_{\Theta\in\Scal}P_{(\Theta,\rho,\sigma)}\left(\left|T-h(\Theta)\right|>\kappa p^{-1/2}\right)>1/4, \label{eq: minimax bnd}
 		\end{equation}
 		where $\kappa>0$ is a constant depending on $(c_{1},c_{3},c_{5})$
 		and $\inf_{T}$ is taken over all measurable functions of the data
 		$(X,Y)$. 
 	\end{enumerate}
 \end{theorem}

 Theorem \ref{thm: minimax sparse} gives the minimax rate \textit{without} SSV and the incoherence condition. It provides a similar intuition to \cite{koltchinskii2020efficient}. For instance, (\ref{eq: minimax sparse functional}) shows that it is impossible to guarantee entrywise consistency for sparse $g$ in the considered setting without SSV or incoherence.

 In addition, Equation (\ref{eq: minimax bnd}) implies that the rate $O_P(p^{-1/2})$ is minimax optimal for estimating dense averages in the absence of SSV and incoherence. This rate of convergence is   slower than that obtained in Theorem \ref{th4.1} which makes use of SSV and incoherence. For instance, in the factor model with a finite number of strong factors, Theorem \ref{th4.1} implies that the rate of convergence can be as fast as $\frac{1}{p}\sum_j\widehat\theta_{ij}-\frac{1}{p}\sum_j \theta_{ij}= O_P\left(\frac{1}{\sqrt{np}}+\frac{1}{p}\right)$.\footnote{This rate holds if the factors have zero mean so that $V_0'g=\frac{1}{p}\sum_{j=1}^pv_j=O_P(p^{-1})$, which is the case for no-intercept factor models. Strictly speaking, this setting was ruled out by Assumption \ref{ass4.3}, which requires $\|V_0'g\|\geq cp^{-1/2}$. However, Assumption \ref{ass4.3} is used only for obtaining the asymptotic distribution. 
 	This assumption can be relaxed when only the rate of convergence is of interest. }

 These minimax results for estimating linear combinations of elements of a low-rank matrix {without SSV and incoherence} are new to the literature. The result closest to ours is \cite{koltchinskii2020efficient} which provides minimax rates for estimating linear functionals of the eigenvectors of low-rank matrices. They show that the minimax optimal rate can be slow if the SSV condition does not hold. Other results on the minimax bounds for learning an eigenspace can be found in \cite{berthet2013optimal}, \cite{birnbaum2013minimax}, and \cite{cai2013sparse}.

 \subsection{Dense functional inference without SSV and Incoherence}\label{sub: dense no SSV}
 
 When  $g$ is a vector of \textit{dense} weights, the second minimax result in Theorem \ref{thm: minimax sparse} suggests that consistency can be achieved without the SSV and incoherence conditions at the cost of a slower rate of convergence. For completeness, we introduce an alternative estimator that could be used when one does not with to impose these assumptions.
 
 Specifically, suppose $g=(g_1,...,g_p)'\in\mathbb R^p$ is a vector of dense weights as defined in \eqref{eq4.1dense}, and we are interested in 
 the functional $h_i(\Theta):=\theta_i'g$. We propose the following 
 estimator in the spirit of inverse probability weighting:
 $$
 \widehat{h_i(\Theta)} = \sum_{j=1}^p\frac{g_j y_{ij}x_{ij}}{\widehat \mu_{j,i}^2},\quad \widehat\mu_{j,i}^2=\frac{1}{n-1}\sum_{k\neq i}x_{kj}^2.
 $$
 Note that this estimator does not require knowing the rank or even that the rank is consistently estimable. 
 It is defined as the weighted average of the $i^{\text{th}}$ row of $Y$ and $X$ with weight proportional to a leave-one-out estimator of the inverse of $\mu_j^2:=\E x_{ij}^2$.

 Let $$W_{ij}: =x_{ij}\varepsilon_{ij}+
 x_{ij}^2\theta_{ij}.$$

 \begin{theorem}\label{th4.2no}
 	Let  $g$ be  dense   in the sense of (\ref{eq4.1dense}), and assume $\E x_{ij} \varepsilon_{ij}=0 $. Suppose $W_{ij}$ is independent over $j$ and that $\E W_{ij}^4<C$, $\E x_{ij}^2>c>0$, and $\Var(W_{ij})>c>0$. In addition, suppose $\sqrt{p}\log p=o(n)$. Then 
 	$$
 	s_n^{-1} \sqrt{p}[ \widehat{h_i(\Theta)}-\theta_i'g]\to^d N(0,1)
 	$$
 	where $s_n^2= p \sum_{j=1}^pg_j^2 (\E x_{ij}^2)^{-2}\Var(W_{ij} ).$
 \end{theorem}

 \section{Application to Heterogeneous Treatment Effects}\label{sec:treatment}

 As an important illustration, we show how to apply our framework in a treatment effects setting. Suppose that, for each time $j=1,...,p$   and each unit $i=1,..., n$, there is a pair of potential outcomes 
 \begin{equation}\label{eq: potential outcomes}
 	Y_{ij}(m) =h_{j,m}(\eta_i) +e_{ij}(m),\quad m\in\{0,1\}.
 \end{equation}
 Here $m$ is denotes treatment $(m=1)$ or control ($m=0$) state. In any time period $j$ and for any unit $i$, we observe either $Y_{ij}(1)$ or $Y_{ij}(0)$, but not both, depending on the unit's realized treatment state in that period. The treatment effect depends on time-varying functions $h_{j,m}(.)$ of unit specific state variable $\eta_i$; both $h_{j, m}(\cdot)$ and $\eta_i$ may be unobservable and random. {For clarity, we focus on the scenario where the goal is to perform statistical inference on a long-run treatment effect for a given unit $i$: 
 	$$
 	\tau_i:= \frac{1}{p}\sum_{j=1}^p\nu_{ij}
 	$$
 	where $\nu_{ij}= h_{j,1}(\eta_i)- h_{j,0}(\eta_i)$ is the treatment effect for unit $i$ at time $j$.}
 

 Define the treatment status indicator 
 $$ x_{ij}(m)=1\{\text{unit $i$ at period $j$ is in state $m$} \}= 1\{Y_{ij}(m) \text{ is observable} \}.
 $$
 Consider the following treatment scenario. Suppose the entire time span $\{1,2,...,p\}$ is divided into two periods,
 $$
 T_0=\{1,...,p_0\} \ \text{and} \ T_1=\{ p_0+1,...,p\},
 $$
 where both $p_0$ and $p_1:=p-p_0$ are large and both periods are known. 
 We assume 
 \begin{eqnarray}\label{eq5.2}
 	\begin{split}
 		\max_{i\leq n}\sum 1\{j\in T_0:x_{ij}(0)=0\} &=& o_P\left(d_{n,p_0}\right),\cr
 		\max_{i\leq n}\sum 1\{j\in T_1:x_{ij}(1)=0\} &= & o_P\left( d_{n, p_1}\right),
 	\end{split}
 \end{eqnarray}
 where $d_{n,p_0}$ and $d_{n,p_1}$ are slowly growing sequences defined in (\ref{eq4.2daerv}). That is, each unit is in the control state during most periods in $T_0$, and each unit is the treatment state during most periods in $T_1$. We thus refer to $T_0$ and $T_1$ respectively as the ``control period'' and the ``treatment period''.
 We refer to this treatment scenario as ``systematic treatment,'' and note that treatment assignments are strongly dependent in this setting, which results in an important difference from much of the literature on inference in matrix completion settings. 
 In terms of our formal conditions, this scenario corresponds to the case of Assumption \ref{ass4.3} (ii).c.

 \subsection{Treatment effect inference}
 Let $ \theta_{ij}(m):=h_{j,m}(\eta_i).$ 
 We can then rewrite the model for potential outcomes \eqref{eq: potential outcomes} as 
 \begin{eqnarray}
 	y_{ij}(0) &=& \theta_{ij}(0) x_{ij}(0) + \varepsilon_{ij}(0),\quad j\in T_0 \label{eq5.120} \\
 	y_{ij}(1) &=& \theta_{ij}(1) x_{ij}(1) + \varepsilon_{ij}(1),\quad j\in T_1 \label{eq5.121} 
 \end{eqnarray}
 where $y_{ij}(m) = Y_{ij}(m) x_{ij}(m)$, and $\varepsilon_{ij}(m)=e_{ij}(m)x_{ij}(m).$ Let $\Theta(m)$ denote the $n\times p$ matrix of $(\theta_{ij}(m))_{n\times p}$. As, e.g., previously note by \cite{athey2018matrix}, it is then clear that recovering elements of $\Theta(m)$ is equivalent to solving a matrix completion problem. 
 
 In Section \ref{sec:low:trea} we provide sufficient conditions to establish that $\Theta(m)$ is an approximate low-rank matrix that satisfies both the SSV and incoherence conditions. Under these conditions, we can then estimate treatment effects by simply applying Algorithm \ref{alg3.1} twice -- once using the data from period $T_0$ and once using the data from period $T_1$. 
 
 \textbf{Step 1: } Apply Algorithm \ref{alg3.1} to (\ref{eq5.120}) to estimate $\Theta(0)$.

 \textbf{Step 2: } Apply Algorithm \ref{alg3.1} to (\ref{eq5.121}) to estimate $\Theta(1)$.

 \textbf{Step 3: } Make inference on the treatment effects from the estimated $\Theta(1)-\Theta(0)$. 
 
 Let $\widehat\theta_{ij}(m)$ denote the $(i,j)$ element of the estimated matrix $\Theta(m)$. The average treatment effect estimator is then given by 
 $$
 \widehat\tau_i:= \frac{1}{p_1}\sum_{j\in T_1} \widehat\theta_{ij}(1) - \frac{1}{p_0}\sum_{j\in T_0} \widehat\theta_{ij}(0).
 $$
 It is straightforward to extend Theorem \ref{th4.1} to this context, which leads to the asymptotic distribution of the estimated treatment effects. Formal results are to be presented in Section \ref{sec:theory:trea}.

 \subsection{The low-rank approximation}\label{sec:low:trea}
 We show that the matrix formed from elements $h_{j, m}(\eta_i)$ can be approximated by a low-rank matrix with slowly growing rank. To aid in focusing on the main idea, we suppress the notation ``$m$'' throughout this section. 
 
 Consider a family of time-varying functions $h_{j}(\cdot)$ of subject-specific latent variables $\eta_i$. Let $\Theta$ be the $n \times p$ matrix obtained by setting the $(i,j)$ element of $\Theta$ to $h_j(\eta_i)$. Suppose $h_j(\cdot)$ has a sieve approximation:
 \begin{equation}\label{eq2.3}
 	h_j(\eta_i) = \sum_{k=1}^J\lambda_{j, k}\phi_k(\eta_i) + r_{ij}
 	=\lambda_j'\Phi_i+r_{ij}
 \end{equation}
 where $\Phi_i:=(\phi_1(\eta_i),...,\phi_J(\eta_i))'\in\mathbb R^J$ is a set of sieve transformations of $\eta_i$ using $\phi_k(\cdot)$ as the basis functions, $\lambda_j=(\lambda_{j,1},...,\lambda_{j,J})'$ is the vector of sieve coefficients for $h_j(\cdot)$, and $r_{ij}$ is the sieve approximation error. Write $\Phi$ as the $n\times J$ matrix of $\Phi_i$, $\Lambda$ as the $p\times J$ matrix of $\lambda_j$, and $R$ as the $n\times p$ matrix of $r_{ij}$. Then the matrix form of (\ref{eq2.3}) is 
 $$
 \Theta = \underbrace{\Phi\Lambda' }_{\Theta_0}+ R.
 $$
 Clearly, $\rank(\Theta_0)\leq J$, and there is a rotation matrix $H$ so that columns of $\Lambda H$ are the right singular-vectors of $\Theta_0$. The error $R$ is naturally present as the sieve approximation error which will decrease as more elements are considered in the sieve approximation. It is then natural to consider sequences where 
 $J$ increases slowly with $(n,p)$.
 
 We now illustrate how both the SSV and incoherence conditions can hold in this setting under sensible conditions on the functional space and the sieve bases. Suppose $h_j$ belongs to a H\"{o}lder class: For some $C, \beta,\alpha>0$,
 $$
 \{h: \max_{\gamma_1+...+\gamma_{d}=\beta}\left|\frac{\partial^\beta h(x)}{\partial x_1^{\gamma_1}...\partial x_{d}^{\gamma_{d}} }-\frac{\partial^\beta h(y)}{\partial y_1^{\gamma_1}...\partial y_{d}^{\gamma_{d}} }\right|\leq C\|x-y\|^\alpha, \text{for all }x,y\}.
 $$
 Further suppose that a common basis, such as polynomials or B-splines are considered. We will then have 
 $$
 \max_{ij}|r_{ij}|\leq C J^{-a},\quad a= (\beta+\alpha)/\dim(\eta_i),
 $$
 which can be made arbitrarily small for sufficiently smooth functions even if $J$ grows slowly.

 Now, suppose there exists a $b \geq 0$ such that $\psi_{J}(\Theta) \leq \psi_1(\Theta) \leq C J^b\psi_J(\Theta)$ for some $C>1$. It is then easy to show that the sequence $\psi_{np}$ can be taken as 
 $$
 \psi_{np}\asymp \sqrt{J^{-(2b+1)}\sum_{i=1}^n\sum_{j=1}^p h_j(\eta_i)^2}.
 $$
 We then have that the top $J$ singular values grow at this rate which leads to the SSV condition.
 
 Finally, write $S_\Lambda=\frac{1}{p} \Lambda'\Lambda$, $S_\Phi= \frac{1}{n}\Phi'\Phi$, and $A =S_{\Phi}^{1/2} S_{\Lambda} S_{\Phi}^{1/2}.$ Also let $G_\Phi$ be a $J\times J$ matrix whose columns are the eigenvectors of $A$, and let $T$ be the diagonal matrix of corresponding eigenvalues. Letting $H_\Phi:= S_\Phi^{-1/2} G_\Phi$, it can be verified that 
 $$
 \Theta_0\Theta_0' \Phi H_\Phi = pn \Phi H_\Phi T \ \text{and} \  \frac{1}{n} ( \Phi H_\Phi)' \Phi H_\Phi= I.
 $$ 
 Thus, the columns of $\frac{1}{\sqrt{n}} \Phi H_\Phi$ are the left singular-vectors of $\Theta_0$, and the eigenvalues of $ npA $ equal the first $J$ eigenvalues of $ \Theta_0'\Theta_0$. Similarly, we can define $ H_\Lambda= S_\Lambda^{-1/2} G_\Lambda$ where $G_\Lambda$ is a $J\times J$ matrix whose columns are the eigenvectors of $S_{\Lambda}^{1/2} S_{\Phi} S_{\Lambda}^{1/2}$.
 Hence, we have
 \begin{equation}\label{eqc.1}
 	U_0 =n^{-1/2} \Phi H_\Phi,\quad V_0= p^{-1/2}\Lambda H_\Lambda.
 \end{equation}
 Thus, 
 \begin{eqnarray}\label{eq5.3}\begin{split}
 		\max_{i\leq n} \|u_i\|&\leq&n^{-1/2} \max_{i\leq n} \|\Phi_i\|\psi_{\min}^{-1/2} (S_{\Phi})\cr
 		\max_{j\leq p} \|v_j\|&\leq& p^{-1/2} \max_{j\leq p} \|\lambda_j\|\psi_{\min}^{-1/2} (S_{\Lambda}).
 	\end{split}
 \end{eqnarray}
 
 It then follows that the incoherence condition holds as long as we can obtain proper upper bounds for $ \max_{j\leq p} \|\lambda_j\|$ and $ \max_{i\leq n} \|\Phi_i\|$. For example, 
 if $\{h_j(\cdot): j\leq p\}$ is further restricted to a Hilbert space with a uniform $L_2$- bound,
 $$
 \max_{j\leq p}\sum_{k=1}^{\infty}\lambda_{j, k}^2<\infty,
 $$
 then $ \max_{j\leq p} \|\lambda_j\|<C$. 
 
 We formalize the preceding discussion in the following assumption and lemma.

 \begin{assumption}\label{asstreat2}
 	(i) $\max_{j\leq J}\sup_{\eta}|\phi_j(\eta)|<C$, $\E\psi_{\min}^{-1}(S_{\Phi})<C$, and $\psi_{\min}^{-1}(S_{\Lambda})<C$. 
 	
 	(ii) The sieve approximation satisfies $$\max_{ij}|r_{ij}|\leq C J^{-a}$$ for some $a>0$.
 	
 	(iii) $\{h_{j}(\cdot): j\leq p\}$ belong to ball $\mathcal H(\mathcal U, \|\|_{L_2}, C)$ inside a Hilbert space spanned by the basis $\{\phi_k: k=1,...\}$
 	with a uniform $L_2$-bound $C$:
 	$$
 	\sup_{h\in \mathcal H(\mathcal U, \|\|_{L_2})} \|h\|\leq C,
 	$$
 	where $\mathcal U$ is the support of $\eta_i$.
 	
 \end{assumption}

 \begin{lemma}\label{lem6.1}  Suppose Assumption \ref{asstreat2} holds. Then
 	
 	(i) The minimum nonzero singular value $\psi_{np} $ for $\Theta_0=\Phi\Lambda'$ can be taken as
 	$$
 	\psi_{np}^2\asymp {J^{-(2b+1)}\sum_{i=1}^n\sum_{j=1}^p h_{j}(\eta_i)^2},\quad m=0,1,
 	$$
 	which means $\psi_J(\Theta_0)\geq c\psi_{np} $ for this choice of $\psi_{np}$.
 	
 	(ii)	The incoherence  Assumption \ref{ass:4.6} holds.
 	
 	(iii) The low-rank approximation error satisfies $\|R\|_{(n)}\leq C(p\vee n)^{3/2} J^{-a} .$
 \end{lemma}

 \subsection{Reproducing kernel representation}
 We now verify the eigengap condition: 	Let $A=  \frac{1}{pn}\Theta\Theta'.$ There are constants $b,d\geq0$ such that 
 \begin{eqnarray}\label{eq5.9}	
 	\begin{split}
 		\psi_1(A)/\psi_{J}(A)&\leq&  O_P(  J^{b})\cr
 		\min_{k=1...J-1}\psi_{k}(A)-	\psi_{k+1}(A)&\geq& c J^{-d}.
 	\end{split}
 \end{eqnarray}

 Below we verify the above conditions   when the treatment functions are generated from a Gaussian process.

 Suppose $\eta_i$ are uniformly generated from $[0,1]$, and functions $h_{j}(\cdot)$ are independently generated from a Gaussian process with covariance kernel 
 $$
 K(\eta_1,\eta_2) = \Cov(h_j(\eta_1), h_j(\eta_2)),
 $$
 where $K(\cdot,\cdot)$ is  a continuous positive semi-definite kernel function  supported on a compact set. In addition, suppose the associated integral operator 
 $$(Tf)(\cdot)=\int K(\cdot, \eta) f(\eta) d\eta$$ is positive semi-definite. Let $\{\bar\phi_k(\cdot)\}$ and $\nu_k\geq0$ be the eigenfunctions and eigenvalues of $T$.
 Then by Mercer's theorem, $\{\bar\phi_k(\cdot)\}$ is an orthonormal basis so that $K$ has the following representation:
 $$
 K(\eta_1,\eta_2)=\sum_{k=1}^{\infty} \nu_k\bar\phi_k(\eta_1)\bar\phi_k(\eta_2),
 $$
 where the infinite sum can be approximated arbitrarily well by finite truncation $J$ as $J\to\infty.$

 Now consider the $n\times n$ matrix $\frac{1}{p}\Theta\Theta'$, whose $(i,l)$ element is 
 $$
 \frac{1}{p} \sum_{j=1}^ph_j(\eta_i)h_j(\eta_l) = K(\eta_i,\eta_l) +o_P(1) =\bar\Phi_i' D_{\lambda} \bar\Phi_l +o_P(1)
 $$
 where $\bar\Phi_i'=(\bar\phi_1(\eta_i),...,\bar\phi_J(\eta_i))$ and $D_{\lambda}$ is a diagonal matrix of $(\nu_1,...,\nu_J)$. Also, because the $h_j$ are independently generated from the Gaussian process, the $o_P(1)$ terms are uniform over all elements.  
 Thus, we have an approximate low-rank representation of $\Theta\Theta'$:
 $$
 \Theta\Theta' =\left[\sum_{j}h_j(\eta_i)h_j(\eta_l)\right]_{n\times n}\approx  p  \bar\Phi D_{\lambda} \bar\Phi' .
 $$

 Because the columns of $\bar\Phi$ are formed from eigenfunctions, its columns are approximately orthonormal bases as eigenvectors of $\Theta\Theta'$. Hence the diagonals  of $D_{\lambda}$ are also approximately the top $J$ eigenvalues of $\frac{1}{np}\Theta\Theta'$. This observation heuristically shows that the top eigenvalues of $\frac{1}{np}\Theta\Theta'$ are approximately the same as those of the integral operator $T$ associated with the reproducing kernel function. 
 
 Rigorously, we can  verify this condition as follows.
 The conditions of the  lemma below are required to hold for both $m\in\{0,1\}$ in our treatment effect setting.
 \begin{lemma}\label{lem5.2} Suppose the eigenvalues of the integral operator $T$ satisfy
 	$$
 	\nu_k=M k^{-\alpha},\quad  k=1,2,...
 	$$
 	for some $M,\alpha>0$. Further, suppose $\sqrt{\frac{\log n}{p}}+r_J+ \frac{J}{\sqrt{n}}=o_P(J^{-\alpha-1})$, where we recall that $R$ is the remainder matrix in (\ref{eq5.8}) and $r_J:=\sup_{\eta_1,\eta_2}|\sum_{k>J}\nu_k\bar\phi_k(\eta_1)\bar \phi_k(\eta_2)|$. Then the eigengap condition  (\ref{eq5.9}) holds. Specifically, let $A=  \frac{1}{pn}\Theta\Theta'$,
 	\begin{eqnarray}	
 		\begin{split}
 			\psi_1(A)/\psi_{J}(A_m)&\leq&  O_P(  J^{\alpha})\cr
 			\min_{k=1...J-1}\psi_{k}(A_m)-	\psi_{k+1}(A)&\geq& c J^{-(\alpha+1)}.
 		\end{split}
 	\end{eqnarray}
 \end{lemma}

 \subsection{Inference for treatment effects under systematic assignment}\label{sec:theory:trea}
 
 Building on the previous subsections, suppose $h_{j, m}(\eta_i) $ has the following sieve representation: 
 $$
 h_{j, m}(\eta_i) = \sum_{k=1}^J\lambda_{j, k, m}\phi_k(\eta_i) + r_{ij}(m),\quad m=0,1.
 $$
 We then have that the matrix $\Theta(m):=(\theta_{ij}(m))_{n\times p_m}$ admits an approximate low-rank structure for each $m\in\{0,1\}$:
 \begin{equation}\label{eq5.8}
 	\Theta(m)= \Theta_0(m)+ R(m),\quad \Theta_0(m)=\Phi \Lambda_m',\quad R(m)=(r_{ij}(m))_{n\times p_m},
 \end{equation}
 where  $\Lambda_m$ is the $p\times J$ matrix of $\lambda_{j,k,m}$. 

 Note that $\widehat\tau_i$ estimating a sensible average treatment effect relies on an additional stability assumption.
 Define, for $m\in\{0,1\}$, 
 $$
 \zeta_{ij}(m):= x_{ij}(m) v_j(m)'\bar B(m)^{-1} \frac{1}{p_m}\sum_{j\in T_m} v_j(m)
 $$ where $\bar B(m)= \sum_{j\in T_m}x_{ij}(m) v_j(m) v_j(m)'. $
 Applying the analysis of Theorem \ref{th4.1}, we have
 \begin{align*}
 	\widehat\tau_i &- \tau_i = \sum_{j\in T_1} e_{ij} \zeta_{ij}(1) - \sum_{j\in T_0} e_{ij} \zeta_{ij}(0)+ 
 	o_P(\min\{p_0, p_1\}^{-1/2}) \\
 	&+
 	\left( \frac{1}{p_1}\sum_{j\in T_1} \theta_{ij}(1) - \frac{1}{p}\sum_{j=1}^p \theta_{ij}(1)\right)
 	-\left( \frac{1}{p_0}\sum_{j\in T_0} \theta_{ij}(0) - \frac{1}{p}\sum_{j=1}^p \theta_{ij}(0)\right) .
 \end{align*}
 This expansion yields the asymptotic distribution of $\widehat\tau_i$ under the condition that the second line on the right-hand-side is bounded by $o_P(\min\{p_0, p_1\}^{-1/2})$. That is, we need stability of treatment and control averages in the sense that the average of $\theta_{ij}(0)$ and $\theta_{ij}(1)$ obtained over the respective subsamples does not deviate too far from the infeasible average that would be obtained looking over the entire sample period.

 \begin{theorem}\label{th5.2} 
 	Suppose Assumptions  \ref{ass4.3}, \ref{ass.5mom}, \ref{ass4.7} hold. Suppose Assumption \ref{asstreat2} holds for $h_{j,0}$ and $h_{j,1}$. In addition, suppose 
 	\begin{align*}
 		\frac{1}{p_1}\sum_{j\in T_1} \theta_{ij}(1) - \frac{1}{p}\sum_{j=1}^p \theta_{ij}(1) &= o_P(\min\{p_0, p_1\}^{-1/2}) \ \text{and} \\
 		\frac{1}{p_0}\sum_{j\in T_0} \theta_{ij}(0) - \frac{1}{p}\sum_{j=1}^p \theta_{ij}(0) &= o_P(\min\{p_0, p_1\}^{-1/2}).
 	\end{align*}
 	Let
 	\begin{eqnarray*}
 		\bar s_{np,i}^2&:=& \sum_{j\in T_0} \Var (e_{ij}|X,\eta) \zeta_{ij}(0)^2 + \sum_{j\in T_1} \Var (e_{ij}|X,\eta) \zeta_{ij}(1)^2 . \end{eqnarray*}
 	Suppose there is a constant $c>0$ so that $\bar s_{np,i}^2\min\{p_0, p_1\}>c$ with probability approaching one.
 	Then as $n, p_0, p_1\to\infty,$
 	$$
 	\frac{ \widehat\tau_i- \tau_i}{ \bar s_{np,i}} \to^d N(0,1).
 	$$
 \end{theorem}

 \section{Simulations}\label{sec:sim}

 We now illustrate the performance of our inferential approach through a small simulation study in the systematic treatment assignment setting. We report results for $n=p=400$.
 
 To generate data, we first divide the period of observation $\{1,...,p\}$ equally into two periods $T_0$ and $T_1$ each consisting of $p_m = p/2$ observation times. 
 To generate $x_{ij}(m)$, we generate $n_i$ integers $j_1...j_{n_i}$ without replacement to form a set $ A_i(m)=\{j_1,..., j_{n_i}\} \subset T_m$. The number $n_i\leq N_0$ is uniformly generated to be less than a predetermined number $N_0 \in \{p_m^{1/2},p_m^{1/3},p_m^{1/4}\}$. We then set
 $$
 x_{ij}(m)= \begin{cases}
 	0& \text{ if } j\in A_i(m)\\
 	1& \text{ if } j\notin A_i(m)
 \end{cases}.
 $$
 Hence, for each unit $i$, $x_{ij}(m)=1$, \textit{with up to $N_0$ exceptions}, throughout period $T_m$ whose total length is $p_m$. In addition, we generate the noise  $\varepsilon_{ij}$  independently across both $(i,j)$ and $\varepsilon_{ij}(m)\sim \mathcal N(0,\sigma_{e}^2)$ for $\sigma_{e}=1$. 
 
 One of the key conditions in this scenario is that the treatment effect should be stable in the sense that $\frac{1}{p } \sum_{j=1}^p\theta_{ij}(m)$  can be well approximated by $\frac{1}{p_m} \sum_{j\in T_m}\theta_{ij}(m)$. We thus consider the simplest possible setting where this condition holds by generating time invariant treatment functions: 
 $$
 h_{0}(\eta_i) = \sum_{k=1}^{\infty} \frac{|W_{k}|}{k^a}\sin(k\eta_i),\quad h_{1}(\eta_i) = \sum_{k=1}^{\infty} \frac{(|W_{k}|+2)}{k^a}\sin(k\eta_i).
 $$
 Here  $\eta_i\sim $ Uniform$[-1,1]$, $W_{k}\sim \mathcal N(0,1)$, and the noise is $e_{ij}\sim \mathcal N(0,1)$. 
 The power parameter $a>1$ quantifies the decay speed of the sieve coefficients.

 In terms of implementation of our procedure, we also need $J$ and $\nu$. We do not attempt to infer the rank $J$ from the data. Rather, we look at estimates based on four pre-specified values of the rank: $J=1,...,4$. We set the parameter $\nu$ for the nuclear-norm penalized optimization through a simple plug-in procedure. Specifically, we set 
 \begin{align}\label{eq: nu}
 	\nu=2.2\bar Q(\|Z\circ  X (m)\|; 0.95)
 \end{align}
 where $\bar Q(W ; q)$ denotes the $q^{\textnormal{th}}$ quantile of a random variable $W$ and $Z$ is an $n\times p_m$ matrix whose elements $z_{ij}$ are generated as $\mathcal N(0,\widehat\sigma_{e}^2)$ independent across $(i,j)$ for some estimated $\widehat\sigma_{e}^2$.\footnote{We set $\widehat\sigma^2_{e}$ by obtaining an initial guess, $\tilde\sigma^2_{e}$, from estimating the simple model 
 	$y_{ij} =  x_{ij}  \theta_i  +\sigma_{e}^{-1}u_{ij}$ where $\Var(u_{ij})=1$. We then obtain an initial solution to the nuclear-norm regularized optimization problem with tuning parameter set as in \eqref{eq: nu} with $z_{ij} \sim N(0,\tilde\sigma_{e}^2)$. Letting $\widetilde\theta_{ij}$ denote the nuclear-norm regularized estimator obtained with this initial tuning. We then set $\widehat\sigma^2_{e}=\frac{1}{np}\sum_{ij} \widetilde \varepsilon_{ij}^2,$ where $\widetilde \varepsilon_{ij}=y_{ij}-x_{ij}\widetilde\theta_{ij}$.} This choice can be motivated as in \cite{belloni2013least} and \cite{chernozhukov2018inference}.
 

 We report simulation coverage probabilities of 95\% confidence intervals for $\tau_1$ formed using estimated standard errors based on 1000 simulation replications in Table \ref{tab: coverage}. Overall, the derived asymptotic distributions seem to provide reasonable approximations to the finite sample distributions under our simulation settings, and the good performance appears quite robust to the choice of $J$ in this simulation.

 \begin{table}[t]
 	
 	\begin{center}
 		\caption{\small  Systematic Assignments. Coverage Probabilities of the treatment effect $\tau_i$.   }
 		\label{tab: coverage}
 		\begin{tabular}{cc|cccccccc}
 			\hline
 			\hline
 			$N_0$&power 	$a$		 &     \multicolumn{8}{c}{$J $} \\
 			&& & 1 & & 2 & & 3 & & 4   \\
 			\hline
 			&& & & & & & &   &   \\
 			$p_m^{1/2}$ &4 & & 0.952  &   &  0.950  & &  0.949  & &  0.948   \\ 
 			&3 &  &   0.947&  & 0.943& &  0.943&  &  0.942   \\ 
 			&2 &  &    0.952&  &   0.950& &    0.948&  &   0.949  \\    
 			$p_m^{1/3}$ &4 &  &    0.950 & &  0.950&  &  0.947 &  &  0.945  \\      
 			&3 &  & 0.948 &  &  0.946& &   0.945&  &  0.943    \\   
 			&2 &  &  0.952 &  &  0.950& &   0.948& &  0.947 \\  
 			$p_m^{1/4}$ &4 &&   0.954 & &  0.952 & &   0.949 &&  0.946   \\  
 			&3 &  &  0.954 &  &  0.952& &   0.951& &  0.950\\  
 			&2&  &  0.955 &  &  0.956& &   0.950& &  0.951\\ 
 			
 			\hline
 			\multicolumn{10}{p{9cm}}{\footnotesize Note: This table reports the simulated coverage probability of 95\% confidence intervals. The rank $J$ equals the sieve dimension used. Power $a$ quantifies  the decay rate of the sieve coefficients $\lambda_{j,k}\sim k^{-a}$. Finally, $N_0$ controls the number of ``exceptions" over time (the maximum number of treated during ``control period" and the maximum number of controlled during ``treatment period".)}
 		\end{tabular}
 	\end{center}
 \end{table}%

 

\newpage

  \bibliographystyle{ims}
\bibliography{liaoBib_newest}

\begin{thebibliography}{33}
\expandafter\ifx\csname natexlab\endcsname\relax\def\natexlab#1{#1}\fi
\expandafter\ifx\csname url\endcsname\relax
  \def\url#1{\texttt{#1}}\fi
\expandafter\ifx\csname urlprefix\endcsname\relax\def\urlprefix{URL }\fi

\bibitem[{Abbe et~al.(2020)Abbe, Fan, Wang and Zhong}]{abbe2020entrywise}
\textsc{Abbe, E.}, \textsc{Fan, J.}, \textsc{Wang, K.} and \textsc{Zhong, Y.}
  (2020).
\newblock Entrywise eigenvector analysis of random matrices with low expected
  rank.
\newblock \textit{Annals of statistics} \textbf{48} 1452.

\bibitem[{Athey et~al.(2018)Athey, Bayati, Doudchenko, Imbens and
  Khosravi}]{athey2018matrix}
\textsc{Athey, S.}, \textsc{Bayati, M.}, \textsc{Doudchenko, N.},
  \textsc{Imbens, G.} and \textsc{Khosravi, K.} (2018).
\newblock Matrix completion methods for causal panel data models.
\newblock Tech. rep., National Bureau of Economic Research.

\bibitem[{Belloni and Chernozhukov(2013)}]{belloni2013least}
\textsc{Belloni, A.} and \textsc{Chernozhukov, V.} (2013).
\newblock Least squares after model selection in high-dimensional sparse
  models.
\newblock \textit{Bernoulli} \textbf{19} 521--547.

\bibitem[{Berthet and Rigollet(2013)}]{berthet2013optimal}
\textsc{Berthet, Q.} and \textsc{Rigollet, P.} (2013).
\newblock Optimal detection of sparse principal components in high dimension.
\newblock \textit{The Annals of Statistics} \textbf{41} 1780--1815.

\bibitem[{Birnbaum et~al.(2013)Birnbaum, Johnstone, Nadler and
  Paul}]{birnbaum2013minimax}
\textsc{Birnbaum, A.}, \textsc{Johnstone, I.~M.}, \textsc{Nadler, B.} and
  \textsc{Paul, D.} (2013).
\newblock Minimax bounds for sparse pca with noisy high-dimensional data.
\newblock \textit{Annals of statistics} \textbf{41} 1055.

\bibitem[{Cai et~al.(2013)Cai, Ma and Wu}]{cai2013sparse}
\textsc{Cai, T.~T.}, \textsc{Ma, Z.} and \textsc{Wu, Y.} (2013).
\newblock Sparse pca: Optimal rates and adaptive estimation.
\newblock \textit{The Annals of Statistics} \textbf{41} 3074--3110.

\bibitem[{Cand{\`e}s and Plan(2010)}]{candes2010matrix}
\textsc{Cand{\`e}s, E.~J.} and \textsc{Plan, Y.} (2010).
\newblock Matrix completion with noise.
\newblock \textit{Proceedings of the IEEE} \textbf{98} 925--936.

\bibitem[{Cand{\`e}s and Recht(2009)}]{candes2009exact}
\textsc{Cand{\`e}s, E.~J.} and \textsc{Recht, B.} (2009).
\newblock Exact matrix completion via convex optimization.
\newblock \textit{Foundations of Computational Mathematics} \textbf{9}
  717--772.

\bibitem[{Chen et~al.(2019{\natexlab{a}})Chen, Chi, Fan, Ma and
  Yan}]{chen2019noisy}
\textsc{Chen, Y.}, \textsc{Chi, Y.}, \textsc{Fan, J.}, \textsc{Ma, C.} and
  \textsc{Yan, Y.} (2019{\natexlab{a}}).
\newblock Noisy matrix completion: Understanding statistical guarantees for
  convex relaxation via nonconvex optimization.
\newblock \textit{arXiv preprint arXiv:1902.07698} .

\bibitem[{Chen et~al.(2019{\natexlab{b}})Chen, Fan, Ma and
  Yan}]{chen2019inference}
\textsc{Chen, Y.}, \textsc{Fan, J.}, \textsc{Ma, C.} and \textsc{Yan, Y.}
  (2019{\natexlab{b}}).
\newblock Inference and uncertainty quantification for noisy matrix completion.
\newblock \textit{arXiv preprint arXiv:1906.04159} .

\bibitem[{Chernozhukov et~al.(2018)Chernozhukov, Hansen, Liao and
  Zhu}]{chernozhukov2018inference}
\textsc{Chernozhukov, V.}, \textsc{Hansen, C.}, \textsc{Liao, Y.} and
  \textsc{Zhu, Y.} (2018).
\newblock Inference for heterogeneous effects using low-rank estimations.
\newblock \textit{arXiv preprint arXiv:1812.08089} .

\bibitem[{Chernozhukov et~al.(2015)Chernozhukov, Hansen and
  Spindler}]{CHS:AnnRev}
\textsc{Chernozhukov, V.}, \textsc{Hansen, C.} and \textsc{Spindler, M.}
  (2015).
\newblock Valid post-selection and post-regularization inference: An
  elementary, general approach.
\newblock \textit{Annual Review of Economics} \textbf{7} 649--688.

\bibitem[{Dray and Josse(2015)}]{dray2015principal}
\textsc{Dray, S.} and \textsc{Josse, J.} (2015).
\newblock Principal component analysis with missing values: a comparative
  survey of methods.
\newblock \textit{Plant Ecology} \textbf{216} 657--667.

\bibitem[{Fan et~al.(2022)Fan, Guo and Zheng}]{fan2022estimating}
\textsc{Fan, J.}, \textsc{Guo, J.} and \textsc{Zheng, S.} (2022).
\newblock Estimating number of factors by adjusted eigenvalues thresholding.
\newblock \textit{Journal of the American Statistical Association} \textbf{117}
  852--861.

\bibitem[{Gross(2011)}]{gross2011recovering}
\textsc{Gross, D.} (2011).
\newblock Recovering low-rank matrices from few coefficients in any basis.
\newblock \textit{IEEE Transactions on Information Theory} \textbf{57}
  1548--1566.

\bibitem[{Hastie et~al.(2015)Hastie, Mazumder, Lee and
  Zadeh}]{hastie2015matrix}
\textsc{Hastie, T.}, \textsc{Mazumder, R.}, \textsc{Lee, J.~D.} and
  \textsc{Zadeh, R.} (2015).
\newblock Matrix completion and low-rank svd via fast alternating least
  squares.
\newblock \textit{The Journal of Machine Learning Research} \textbf{16}
  3367--3402.

\bibitem[{Iwakura and Okui(2014)}]{iwakura2014asymptotic}
\textsc{Iwakura, H.} and \textsc{Okui, R.} (2014).
\newblock Asymptotic efficiency in factor models and dynamic panel data models.
\newblock \textit{Available at SSRN 2395722} .

\bibitem[{Jain et~al.(2013)Jain, Netrapalli and Sanghavi}]{jain2013low}
\textsc{Jain, P.}, \textsc{Netrapalli, P.} and \textsc{Sanghavi, S.} (2013).
\newblock Low-rank matrix completion using alternating minimization.
\newblock In \textit{Proceedings of the forty-fifth annual ACM symposium on
  Theory of computing}.

\bibitem[{Jankova and Van De~Geer(2018)}]{jankova2018semiparametric}
\textsc{Jankova, J.} and \textsc{Van De~Geer, S.} (2018).
\newblock Semiparametric efficiency bounds for high-dimensional models.
\newblock \textit{The Annals of Statistics} \textbf{46} 2336--2359.

\bibitem[{Jankov{\'a} and van~de Geer(2021)}]{jankova2021biased}
\textsc{Jankov{\'a}, J.} and \textsc{van~de Geer, S.} (2021).
\newblock De-biased sparse pca: Inference for eigenstructure of large
  covariance matrices.
\newblock \textit{IEEE Transactions on Information Theory} \textbf{67}
  2507--2527.

\bibitem[{Keshavan et~al.(2010)Keshavan, Montanari and Oh}]{keshavan2010matrix}
\textsc{Keshavan, R.~H.}, \textsc{Montanari, A.} and \textsc{Oh, S.} (2010).
\newblock Matrix completion from a few entries.
\newblock \textit{IEEE transactions on information theory} \textbf{56}
  2980--2998.

\bibitem[{Klopp(2014)}]{klopp2014noisy}
\textsc{Klopp, O.} (2014).
\newblock Noisy low-rank matrix completion with general sampling distribution.
\newblock \textit{Bernoulli} \textbf{20} 282--303.

\bibitem[{Koltchinskii et~al.(2020)Koltchinskii, L{\"o}ffler and
  Nickl}]{koltchinskii2020efficient}
\textsc{Koltchinskii, V.}, \textsc{L{\"o}ffler, M.} and \textsc{Nickl, R.}
  (2020).
\newblock Efficient estimation of linear functionals of principal components.
\newblock \textit{The Annals of Statistics} \textbf{48} 464--490.

\bibitem[{Koltchinskii et~al.(2011)Koltchinskii, Lounici and
  Tsybakov}]{koltchinskii2011nuclear}
\textsc{Koltchinskii, V.}, \textsc{Lounici, K.} and \textsc{Tsybakov, A.~B.}
  (2011).
\newblock Nuclear-norm penalization and optimal rates for noisy low-rank matrix
  completion.
\newblock \textit{The Annals of Statistics} \textbf{39} 2302--2329.

\bibitem[{Negahban and Wainwright(2011)}]{negahban2011estimation}
\textsc{Negahban, S.} and \textsc{Wainwright, M.~J.} (2011).
\newblock Estimation of (near) low-rank matrices with noise and
  high-dimensional scaling.
\newblock \textit{The Annals of Statistics} \textbf{39} 1069--1097.

\bibitem[{Neyman(1959)}]{Neyman59}
\textsc{Neyman, J.} (1959).
\newblock Optimal asymptotic tests of composite statistical hypotheses.
\newblock In \textit{Probability and Statistics: The Harald Cram{\'e}r Volume}
  (U.~Grenander, ed.). Almqvist and Wiksell, 213--234.

\bibitem[{Onatski(2010)}]{onatski2010determining}
\textsc{Onatski, A.} (2010).
\newblock Determining the number of factors from empirical distribution of
  eigenvalues.
\newblock \textit{The Review of Economics and Statistics} \textbf{92}
  1004--1016.

\bibitem[{Recht(2011)}]{recht2011simpler}
\textsc{Recht, B.} (2011).
\newblock A simpler approach to matrix completion.
\newblock \textit{Journal of Machine Learning Research} \textbf{12} 3413--3430.

\bibitem[{Rohde and Tsybakov(2011)}]{rohde2011estimation}
\textsc{Rohde, A.} and \textsc{Tsybakov, A.~B.} (2011).
\newblock Estimation of high-dimensional low-rank matrices.
\newblock \textit{The Annals of Statistics} \textbf{39} 887--930.

\bibitem[{Sun and Luo(2016)}]{sun2016guaranteed}
\textsc{Sun, R.} and \textsc{Luo, Z.-Q.} (2016).
\newblock Guaranteed matrix completion via non-convex factorization.
\newblock \textit{IEEE Transactions on Information Theory} \textbf{62}
  6535--6579.

\bibitem[{Sun and Zhang(2012)}]{sun2012calibrated}
\textsc{Sun, T.} and \textsc{Zhang, C.-H.} (2012).
\newblock Guaranteed minimum-rank solutions of linear matrix equations via
  nuclear norm minimization.
\newblock \textit{Advances in Neural Information Processing Systems}  863--871.

\bibitem[{Xia and Yuan(2019)}]{xia2019statistical}
\textsc{Xia, D.} and \textsc{Yuan, M.} (2019).
\newblock Statistical inferences of linear forms for noisy matrix completion.
\newblock \textit{arXiv preprint arXiv:1909.00116} .

\bibitem[{Zhu et~al.(2019)Zhu, Wang and Samworth}]{zhu2019high}
\textsc{Zhu, Z.}, \textsc{Wang, T.} and \textsc{Samworth, R.~J.} (2019).
\newblock High-dimensional principal component analysis with heterogeneous
  missingness.
\newblock \textit{arXiv preprint arXiv:1906.12125} .

\end{thebibliography}

\end{document}